\documentclass[12pt,reqno]{amsart}

\usepackage{graphicx}
\usepackage{mathrsfs}
\usepackage{amsthm}
\usepackage{amssymb}
\usepackage{amsmath,amsthm,amsfonts}
\usepackage{mathabx}
\usepackage{hhline}
\usepackage{textcomp}
\usepackage{amsmath,amscd}
\usepackage[all,cmtip]{xy}
\usepackage{mathtools}
\usepackage[margin=1in]{geometry}
\usepackage{lipsum}
\usepackage{extarrows}
\usepackage{upgreek}
\usepackage{bbm}
\usepackage[displaymath]{lineno}
\usepackage[singlespacing]{setspace}%
\usepackage[colorlinks=true,breaklinks=true,bookmarks=true,urlcolor=blue,
citecolor=blue,linkcolor=blue,bookmarksopen=false,draft=false]{hyperref}
\providecommand{\U}[1]{\protect\rule{.1in}{.1in}}

\theoremstyle{definition}
\newtheorem{theorem}{Theorem}[section]

\newtheorem{prop}{Proposition}[section]

\newtheorem*{theorem*}{Theorem}
\newtheorem{cor}{Corollary}[section]

\numberwithin{equation}{section}

\DeclareMathAlphabet{\mathpzc}{OT1}{pzc}{m}{it}

\newcommand{\nbd}{neighborhood \,}
\newcommand{\li}{Lipschitz}
\newcommand{\rr}{\mathbb{R}}
\newcommand{\nn}{\mathbb{N}}

\newcommand{\uu}{\mathcal{U}}
\newcommand{\vv}{\mathcal{V}}
\newcommand{\ww}{\mathcal{W}}
\newcommand{\w}{\mathbbmss{w}}
\newcommand{\x}{\mathbbmss{x}}
\newcommand{\xx}{\mathcal{X}}
\newcommand{\ab}{\mathcal{A}}

\newcommand{\id}{\mathrm{Id}}
\newcommand{\pr}{\mathrm{Pr}}

\newcommand{\f}{Fr\'{e}chet }
\newcommand{\mc}{$MC^k$}
\newcommand{\bl}[1] {\mathbf {#1}}
\DeclareMathOperator{\dd}{D}
\DeclareMathOperator{\codim}{codim}
\DeclareMathOperator{\Ind}{Ind}
\DeclareMathOperator{\Img}{Img}

\DeclareMathOperator{\Aut}{Aut}

\newcommand\opn{\ensuremath{\mathrel{\mathpalette\opncls\circ}}}

\newcommand{\opncls}[2]{
	\ooalign{$#1\subseteq$\cr
		\hidewidth\raisefix{#1}\hbox{$#1{\stylefix{#1}#2}\mkern2mu$}\cr}}
\def\raisefix#1{
	\ifx#1\displaystyle
	\raise.39ex
	\else
	\ifx#1\textstyle
	\raise.39ex
	\else
	\ifx#1\scriptstyle
	\raise.275ex
	\else
	\raise.150ex
	\fi
	\fi
	\fi
}
\def\stylefix#1{
	\ifx#1\displaystyle
	\scriptstyle
	\else
	\ifx#1\textstyle
	\scriptstyle
	\else
	\ifx#1\scriptstyle
	\scriptscriptstyle
	\else
	\scriptscriptstyle
	\fi
	\fi
	\fi
}
\DeclareFontFamily{U}{mathx}{\hyphenchar\font45}
\DeclareFontShape{U}{mathx}{m}{n}{
	<5> <6> <7> <8> <9> <10>
	<10.95> <12> <14.4> <17.28> <20.74> <24.88>
	mathx10
}{}

\newcommand{\fr}{Fr\'{e}chet }

\newcommand{\Z}{\mathbb{Z}}

\newcommand{\set}[1]{\left\{#1\right\}}
\newcommand{\snorm}[2][]{\left\lVert#2\right\rVert_{#1}}

\makeatletter
\@namedef{subjclassname@2020}{%
	\textup{2020} Mathematics Subject Classification}
\makeatother
\begin{document}

\title{Some applications of transversality for infinite dimensional manifolds}

\author{Kaveh Eftekharinasab}
%\thanks{The author would like to thank the reviewer for the valuable comments  }
\address{Topology lab.  Institute of Mathematics of National Academy of Sciences of Ukraine, Tereshchenkivska st. 3, Kyiv, 01601 Ukraine}

\email{kaveh@imath.kiev.ua}


%    General info
\subjclass[2020]{57N75,  % General position and transversality
	58B15,  %15 Fredholm structures on infinite-dimensional	manifolds
	47H11. %  Degree theory for nonlinear operators 
}

%\date{}

\keywords{Transversality, degree of nonlinear Fredholm mappings, \f manifolds }

\begin{abstract}
   We present some transversality results for a category of  \f manifolds, the so-called \mc-\f manifolds. In this context,
   we apply the obtained transversality results to construct the degree of nonlinear Fredholm mappings
   by virtue of which we prove a rank theorem, an invariance of domain theorem and a Bursuk-Ulam type theorem. 
\end{abstract}

\maketitle
This paper is devoted to the development of transversality and its applications to degree theory of nonlinear Fredholm mappings  
for non-Banachable \fr manifolds. The elaboration is mostly, but not entirely, routine; we shall discuss the related issues.

In attempting to develop transversality  to \fr manifolds we face the following drawbacks which are related to
lack of a suitable topology on a space of continuous linear maps:
\begin{enumerate} \label{isu}
	\item[(1)] In general, the set of isomorphisms between \f spaces is not open in the space of continuous linear mappings.
	\item[(2)]  In general, the set of Fredholm operators between \f spaces is not open in the space of continuous linear mappings.
\end{enumerate} 
Also, a key point in the proof of an infinite dimensional version of Sard's theorem is that a Fredholm mapping $ \varphi $ near origin has a local representation of the form $ \varphi (u,v) = (u, \eta(u,v)) $ for some smooth mapping $ \eta $; 
indeed, this is a consequence of an inverse function theorem. 
      
To obtain a version of Sard's theorem for  \fr manifolds, based on the ideas of M\"{u}ller~\cite{m}, it was proposed by the author (\cite{k1}) to consider Fredholm operators which are
\li \,on their domains. There is an appropriate metrizable topology on a space of \li\, linear mappings so that
if we employ this space instead of a space of continuous linear mappings, the mentioned openness issues 
and the problem of stability of Fredholm mappings under small perturbation can be resolved. Furthermore, 
for mappings belong to a class of differentiability, bounded or \mc-differentiability which is introduced in~\cite{m}, 
a suitable version of an inverse function theorem is available,~\cite[Theorem 4.7]{m}.  

An example of \li-Fredholm mapping of class \mc \,can be found in~\cite{j}, where the Sard's theorem~\cite[Theorem 4.3]{k1} is applied to classify all the holomorphic functions locally definable; this gives the additional motivation to study further applications of Sard's theorem.

In this paper, first we improve the transversality theorem~\cite[Theorem 4.2]{k2} by considering all mappings of class \mc, then use it to prove the parametric transversality theorem. Then, for  \li-Fredholm mappings of class \mc\,  we apply the transversality theorem 
to construct the  degree (due to Cacciappoli, Shvarts and Smale), which is defined as 
the group of non-oriented cobordism class of $ \varphi^{-1}(q) $ for some regular value $ q $.

We then prove a rank theorem for  \li-Fredholm mappings of class \mc\,, and use it  to prove an invariance of domain theorem and a Fredholm alternative theorem. Also, using the parametric transversality theorem we obtain a  Bursuk-Ulam type theorem.

\section{Bounded Fr\'{e}chet manifolds}
In this section,  we shall briefly recall the basics of \mc-\fr manifolds  for the convenience of readers, which also allows us to  establish our notations for the rest of the paper. For more studies, we refer to~\cite{k1,k2}.
 
 Throughout the paper we assume that  $E, F$ are Fr\'{e}chet spaces and $CL(E,F)$ is the space of all continuous linear mappings from $E$ to $F$ topologized by the compact-open topology. If $ T $ is a topological space by $ U \opn T $ we mean $ U $
 is open in $ T $.
 
 Let $\varphi: U \opn E \to F$  be a continuous map. If the directional (G\^{a}teaux) derivatives
$$\dd \varphi(x)h = \lim_{ t \to 0} \dfrac{\varphi(x+th)-\varphi(x)}{t}$$
exist for all $x \in U$ and all $ h \in E $, and  the induced map  $\dd \varphi(x) : U \rightarrow CL(E,F)$ is continuous for all
$x \in U$, then  we say that $ \varphi $ is a Keller's differentiable map of class  $C^1_c$. 
The higher directional derivatives and $C^k_c$-mappings, $k\geq2$, are defined in the obvious inductive fashion.

To define bounded or \mc-differentiability, we endow a Fr\'{e}chet space $ F$ with a translation
invariant metric $\varrho$ defining its topology, and then introduce the metric concepts which strongly depend on
the choice of $\varrho$.  We consider only metrics of the following form
$$
\varrho (x,y) = \sup_{n \in \nn} \dfrac{1}{2^n} \dfrac{\snorm[F,n] {x-y}}{1+\snorm[F,n] {x-y}},
$$
where $ \snorm[F,n]{\cdot} $ is a collection of seminorms generating the  topology of $F$.

Let $\sigma$ be a metric that defines the topology of a  Fr\'{e}chet space $ E $. Let $\mathbb{L}_{\sigma,\varrho}(E,F)$ be the set of all 
linear mappings $ L: E \rightarrow F $ which are (globally) Lipschitz continuous as mappings between metric spaces $E$ and $F$, that is 
\begin{equation*}
\mathpzc{Lip} (L )\, \coloneq \displaystyle \sup_{x \in E\setminus\{0_F\}} \dfrac{\varrho (L(x),0_F)}{\sigma( x,0_F)} < \infty,
\end{equation*}
where $\mathpzc{Lip}(L)$ is the (minimal) Lipschitz constant of $L$.

The translation invariant metric 
\begin{equation} \label{metric}  
\mathbbm{d}_{\sigma,\varrho}: \mathbb{L}_{\sigma,\varrho}(E,F) \times \mathbb{L}_{\sigma,\varrho}(E,F) \longrightarrow [0,\infty) , \,\,
(L,H) \mapsto \mathpzc{Lip}(L-H)_{\sigma,\varrho} \,,
\end{equation}
on $\mathbb{L}_{\sigma,\varrho}(E,F)$ turns it  into an Abelian topological group. We always topologize the space $\mathbb{L}_{\sigma,\varrho}(E,F)$ by the metric~\eqref{metric}.

Let $ \varphi:U \opn E \rightarrow F $ be
a continuous map. If $\varphi$ is Keller's differentiable, $ \dd \varphi(x) \in \mathbb{L}_{\sigma,\varrho}(E,F) $ for all $ x \in U $ and the induced map 
$ \dd \varphi(x) : U \rightarrow \mathbb{L}_{\sigma,\varrho}(E,F)$ is continuous, then $ \varphi $ is called bounded differentiable 
or $ MC^{1} $ and we write  $\varphi^{(1)} = \varphi' $. We define  for $k>1 $  mappings of class $ MC^k$, recursively. 

An $MC^k$-Fr\'{e}chet manifold is a Hausdorff second countable topological space modeled on a Fr\'{e}chet space with an atlas of coordinate 
charts  such that the coordinate transition functions are all
$ MC^{k} $-mappings. We define $MC^k$-mappings between Fr\'{e}chet manifolds as usual.
Henceforth, we assume that $ M $ and $ N $ are connected \mc-\f manifolds modeled on \f spaces 
$ (F, \varrho) $  and $ (E,\sigma) $, respectively. 

 A mapping $ \varphi \in \mathbb{L}_{\sigma,\varrho}(E,F)$ 
is called Lipschitz-Fredholm operator if its kernel has finite dimension and its image  has finite co-dimension. 
The index of $ \varphi $ is defined by 
$$
\Ind \varphi = \dim \ker \varphi - \codim \Img \varphi.
$$
We denote by $ \mathcal{LF}(E,F) $  the set of all Lipschitz-Fredholm operators, and by $ \mathcal{LF}_l(E,F) $ the subset of $ \mathcal{LF}(E,F) $  
consisting of those operators  of index $l$.

An \mc-Lipschitz-Fredholm mapping $ \varphi:M \rightarrow N ,\, k\geq1 $, is a mapping such that for each $ x \in M $
the derivative $ \dd \varphi(x): T_xM  \longrightarrow T_{f(x)}N$ is a Lipschitz-Fredholm operator. 
The index of $ \varphi$, denoted by $\Ind{\varphi}$, is defined to be the index of $ \dd \varphi (x) $ for some $ x $ which does not depend on the choice of $ x $, see~\cite[Definition 3.2 ]{k1}.

Let $\varphi: M \rightarrow N$ $(k\geq1)$ be an \mc-mapping. We denote by $T_x\varphi: T_xM \rightarrow T_{\varphi(x)}N$
the tangent map of $f$ at $x \in M$ from the tangent space $T_xM$ to the tangent space $T_{\varphi(x)}N$.
We say that $\varphi$ is an immersion (resp. submersion)
provided $T_x\varphi$ is injective (resp. surjective) and the range $\Img(T_x\varphi)$ (resp. the kernel $\ker(T_x\varphi) $)
splits in $T_{\varphi(x)}N$ (resp. $T_xM$) for any $x \in M$. 
An injective immersion $f: M \rightarrow N$ which gives an isomorphism onto a submanifold
of $N$ is called an embedding. A point $ x \in M $ is called a regular point  if 
$ \operatorname{D}f(x): T_xM  \longrightarrow T_{f(x)}N $ is surjective. The corresponding
value $f(x)$ is a regular value. Points and values other than regular are called critical points and values, respectively.

	Let $\varphi:M \to N$ be an \mc-mapping, $ k \geq 1 $.
	We say that $\varphi$ is transversal to a submanifold $S \subseteq N$ and write $\varphi \pitchfork S$ if
	either $ \varphi^{-1}(S) = \emptyset $, or if for each $ x \in \varphi^{-1}(S) $ 
	\begin{enumerate}
		\item $(T_x \varphi)(T_x M) + T_{\varphi(x)}S = T_{\varphi(x)}N$, and
		\item $ (T_x \varphi)^{-1}(T_{\varphi(x)}S) $ splits in $ T_xM $.
	\end{enumerate}

In terms of charts, $ \varphi \pitchfork S $ when $ x \in \varphi^{-1}(S) $
there exist charts $ (\phi,\uu) $ around $ x $ and $ (\psi,\vv) $ around $ \varphi (x) $ such that
$$
\psi: \vv \to \vv_1 \times \vv_2 
$$
is an \mc-isomorphism on a product, with 
$$
\psi(\varphi(x)) = (0_E,0_E)\, \quad \varphi (S \cap \vv) =\vv_1 \times \set{0_E}.
$$
Then the composite mapping
$$
\uu \xrightarrow {\varphi} \vv \xrightarrow{\psi}\vv_1 \times \vv_2 \xrightarrow{\pr_{V_2}} \vv_2.
$$
is an \mc-submersion, where $ \pr_{V_2} $ is the projection onto $ \vv_2 $.

\section{Transversality theorems}
We generalize \cite[Theorem 4.2]{k2} and \cite[Corollary 4.1]{k2} for not necessarily  \li-Fredholm  mappings and finite dimensional submanifolds. We shall need the following version of the inverse function theorem
for \mc-mappings.
\begin{theorem}\cite[Theorem 4.7]{m}  \label{invr}
	 Let $\uu \opn  E$, $u_0 \in \uu$ and
	$ \varphi : \uu \rightarrow E $  an $ MC^k$-mapping, $ k \geq 1 $.
	If $\varphi'(u_0) \in \Aut{(E)}  $, then there exists 
	$ \vv \opn \uu $ of $ u_0 $ such that $ \varphi(\vv) $ is open in $ E $ and $ \varphi \vert_\vv : \vv \to \varphi(\vv) $ 
	is an $ MC^k$- diffeomorphism.
\end{theorem}
\begin{prop}\label{p1}
	Let $ \varphi : M \to N $ be an \mc-mapping, $ S \subset N $ an \mc-submanifold and $ x \in \varphi^{-1}(S) $.
	Then $ \varphi \pitchfork S $ if and only if there are charts 
	$ (\uu, \phi) $ around $ x $ with $ \phi(x) = 0_E $ and $ (\vv,\psi) $ around $ y=\varphi (x) $ in $ S $ with
	$ \psi(y) = 0_F $ such that the following hold:
	\begin{enumerate}
		\item[(1)] There are subspaces $ \bf E_1 $ and $ \bf E_2 $ of $ E $, and $ \bf F_1 $ and $ \bf F_2 $ of $ F $
		such that $ E =  \bf E_1 \oplus \bf E_2 $ and $ F =\bf F_1 \oplus F_2 $. Moreover, $\psi( S \cap \vv) = F_1$
		and  
		\begin{align*}
		\phi (\uu) =  E_1 +  E_2  \subseteq {\bf E_1} \oplus {\bf E_2} \\ \nonumber
		 \psi (\vv) =  F_1 +  F_2  \subseteq {\bf F_1} \oplus {\bf F_2}, \nonumber
		\end{align*}
		where $ 0_E \in E_i \opn {\bf E_i} $ and $ 0_F \in F_i \opn {\bf F_i} $, $ i=1,2 $.
		\item[(2)] In the charts the local representative of $ \varphi $ has the form
		\begin{equation}\label{lp}
		\varphi_{\phi\psi} = \overline{\varphi} + \hat{\varphi} \circ \pr_{E_2},
		\end{equation} 
		where $ \overline{\varphi} : E_1 + E_2 \to F_1 $ is an \mc-mapping, $ \hat{\varphi} $
		is an \mc-isomorphism of $ \bf E_2 $ onto $ \bf F_2 $ and $ \pr_{E_2}: E \to {\bf E_2} $
		is the projection.
	\end{enumerate}
\end{prop}
\begin{proof}
{\bf Sufficiency:}	 Let $ (\uu, \phi) $ and $ (\vv, \psi) $ be charts that satisfy the assumptions
we will prove $ \varphi \pitchfork S. $

In the charts, by using the identifications $ T_xM \simeq E, \, T_y \simeq F $, the tangent map
$ T_x\varphi: T_xM \to T_y N $ has the representation 
$$
T_x\varphi = \varphi'_{\phi \psi}(0_E) : E \to F.
$$ 
Also, we have the identification $ T_y S \simeq {\bf F_1} $.

Let $ \pr_{F_2} : F \to {\bf F_2}$ be the projection onto $ {\bf F_2} $.
Since $ \varphi'_{\phi\psi}(0_E) = (\overline{\varphi})' + \hat{\varphi} \circ \pr_{E_2}$ and
$ (\overline{\varphi})'(0) : E \to {\bf F_1} $, it follows that for all 
$ e \in E,\, e=e_1+e_2 \in {\bf E_1} \oplus { \bf E_2 } $
$$
\varphi'_{\phi\psi}(0_E)e = (\overline{\varphi})'(0_E)e + \hat{\varphi} \circ \pr_{E_2}.
$$
Thus, $ \pr_{F_2}  \circ \varphi'_{\phi\psi}(0_E)e = \pr_{F_2} \circ \hat{\varphi} \circ \, \pr_{F_2} (e)$
which means
\begin{equation}\label{key}
\pr_{F_2}  \circ \varphi'_{\phi\psi}(0_E) = \pr_{F_2} \circ \hat{\varphi} \circ \, \pr_{F_2}
\end{equation}
it is a surjective mapping of $ E $  onto $ {\bf F_2} $.

Moreover, we have 
\begin{align*}
	\ker ( \pr_{F_2} \circ \varphi'_{\phi\psi}(0_E)) &= \varphi'_{\phi\psi}(0_E)^{-1} (\pr_{F_2}(0_E))
	= \varphi'_{\phi\psi}(0_E)^{-1} (F_1) \\
	&= \set{e = e_1+e_2 \in {\bf E_1} \oplus {\bf E_2} \mid (\overline{\varphi})'(0_E)e + \hat{\varphi}(e_2) \in \mathbf{F}_1 } \\
		&= \set{e = e_1+e_2 \in {\bf E_1} \oplus {\bf E_2} \mid \hat{\varphi}(e_2) =0_F } \\
		&= \set{e = e_1+e_2 \in {\bf E_1} \oplus {\bf E_2} \mid e_2 =0_E } \\
		&=\mathbf{E}_1.
	\end{align*}
	Which is an \mc- splitting in $ E $ with a component $ E_2 $. From~\eqref{key} it follows that
	$$
	\pr_{F_2} \circ \varphi'_{\phi\psi}(0_E) = \hat{\varphi} : \pr_{E_2} \to \pr_{F_2}
	$$
	which is an \mc-isomorphism.
	
	{\bf Necessity:} Suppose $ \varphi \pitchfork S $. Since $ S $ is an \mc-submanifold of $ N $
	and $ y= \varphi(x) \in S$, there is a chart $(\ww, \w)  $ around $ y $
	having the submanifold property for $ S $ in $ N $: 
	\begin{align*}
		\w(\ww) = W_1+ W_2 \subset \mathbf{F_1} \oplus \mathbf{F_2} = F,\\
		\w(S \cap \ww) = W_1 \subset F, \quad \w (y) =0_E.
	\end{align*} 
Also, there is a chart $ (\xx,\x)  $ around $ x $ such that $ \x(x) =0_F,\, \varphi(\xx) \subset \ww$ and
$$
\varphi_{\x\w} : \x(\xx) \opn E \to \w(\ww) \opn F
$$
is of class \mc. It follows that $ \varphi_{\x\w}(0_E) \circ \pr_{F_2}: E \to \mathbf{F_2} $
is an \mc-submersion as $ \varphi \pitchfork S $. That is, $ \varphi_{\x\w} \circ \pr_{F_2} $
and $ \mathbf{E_1} \coloneq \varphi_{\x\w}'(0_E)^{-1}(\mathbf{F_1}) $ splits in $ E $ with the complement
$ \mathbf{E_2} $ such that
$$
\eta \coloneq \pr_{F_2} \circ \varphi_{\x\w}(0_E) \mid_{\mathbf{E_2}} : \mathbf{E_2} \to \mathbf{F_2}  
$$ 	
is an \mc-isomorphism. Set $ \tau \coloneq (\pr_{E_1} + \eta^{-1} \circ \pr_{F_2} \circ \varphi_{\x\w}) : \x(\xx) \to \w(\ww) $, then $ \tau $ is an \mc-mapping and $ \tau(0_E) =0_E $ and
 $\tau' (0_E)  = (\pr_{E_1} + \eta^{-1} \circ \pr_{F_2} \circ \varphi_{\x\w})'(0_E) = \pr_{E_1}+ \pr_{E_2} =\id_{E}$.
 Because, for all $ e=e_1+e_2 \in \mathbf{E_1}\oplus \mathbf{E_2} $ we have 
 $ (\varphi_{\x\w})'(0_E)e = (\varphi_{\x\w})'(0_E)e_1 + (\varphi_{\x\w})'(0_E)e_2 $. Whence,
 $ \pr_{F_2} \circ (\varphi_{\x\w})'(0_E)e = \pr_{F_2} \circ (\varphi_{\x\w})'(0_E)e_2 = \tau (e_2)$, hence,
 $$ \tau^{-1} \circ \pr_{F_2} \circ (\varphi_{\x\w})'(0_E)e = \tau^{-1}(\tau(e_2)) =\pr_{E_2} .$$  
 
 By the inverse mapping theorem \ref{invr}, $ \tau $ is a local \mc-diffeomorphism.
 
  Assume 
 $ 0_E \in \xx_1 \opn \x(\xx) $ is small enough. Let 
 $$
 \x_1 : \xx_1 \to 0_E \in \xx_2 \opn E
 $$
 be an \mc-diffeomorphism such that 
 \begin{equation}\label{e1}
 \tau \circ \x_1^{-1} = \id_F.
 \end{equation}
 Thus,
 \begin{equation}\label{e2}
 \pr_{F_2} \circ \tau \circ \varphi_{\x\w} = \eta \circ \pr_{E_2}.
 \end{equation}
 If, $ e= e_1+e_2 \in \x_1(\xx_1) $ and $ \x_1^{-1}(e) = \bar{e_1} + \bar{e_2} $, then by \eqref{e1} and \eqref{e2}
 we obtain 
 \begin{align*}
 	\tau \circ \x_1^{-1} (e) &= \tau(\bar{e_1} + \bar{e_2}) \\
 	                    &= \bar{e_1} + \eta^{-1} \circ \pr_{F_2} \circ \varphi_{\x\w} (\bar{e_1} + \bar{e_2})\\
 	                    &=e_1+e_2.
 \end{align*}
 Therefore, $ \bar{e_1} = e_1 $ and $ \tau^{-1} \circ \pr_{F_2} \circ \varphi_{\x\w} = e_2 $
 and so $$ \pr_{F_2} \circ \varphi_{\x\w}(\bar{e_1}+\bar{e_2}) = \tau(e_2) = \tau \circ \pr_{E_2} (e_1+e_2) .$$
 This means, $\pr_{E_2} \circ \varphi_{\x\w} \circ \x_1^{-1}(e) = \eta \circ \pr_{E_2}(e)$ for all $ e \in \x_1(\xx_1) $. Now, define
 \begin{align*}
 	&\phi \coloneq \x_1 \circ \x, \quad \uu \coloneq \x^{-1}(\xx), \\
 	&\psi \coloneq \w,  \quad \vv \coloneq \mathrm{small \,enough \,\nbd of}\, y \, \mathrm{in} \ww.  
 \end{align*}
 Then, $ (\phi, \uu) $ and $ (\psi, \vv) $ are the desired charts. Indeed, 
 \begin{align*}
 	\varphi_{\phi\psi} &= \w \circ \varphi \circ (\x_1 \circ \x) \\
 	               &= \w \circ \varphi \circ \x^{-1} \circ \x^{-1} = \varphi_{\x\w} \circ \x_1^{-1}. 
 \end{align*}
 Thus, 
 \begin{align*}
 	\varphi_{\phi\psi} &= \pr_{F_1} \circ \varphi_{\x\w} \circ \x_1^{-1} + 
 	\pr_{F_2} \circ \varphi_{\x\w} \circ \x_1^{-1}\\
 	&=  \overline{\varphi} + \hat{\varphi} \circ \pr_{E_2},
 \end{align*}
 if we set $ \hat{\varphi} \coloneq \eta $ and $ \overline{\varphi} \coloneq \pr_{E_1} \circ \varphi_{\x\w} \circ \x_1^{-1} $.
\end{proof}
\begin{theorem}[Transversality Theorem]\label{t1}
		Let $ \varphi : M \to N $ be an \mc-mapping, $k\geq1$, $ S \subset N $ an \mc-submanifold and $ \varphi \pitchfork S $. Then, $ \varphi^{-1}(S) $ is either empty of \mc-submanifold of $ M $
		with
		$$
		(T_x \varphi)^{-1}(T_yS) = T_x (\varphi^{-1}(S)), \, x \in \varphi^{-1}(S),\, y=\varphi(x).
		$$
		If $ S $ has finite co-dimension in $ N $, then $\codim (\varphi^{-1}(S)) = \codim S  $. Moreover, 
		if $\dim S = m <\infty$  and $\varphi$ is an \mc-\li-Fredholm mapping of index $ l $, 
		then $\dim \varphi^{-1}(S) = l+m$. 
\end{theorem}
\begin{proof}
	Let $ x \in \varphi^{-1}(S) $, then by Proposition~\ref{p1} there are chart $ (\phi,\uu) $ around $ x $ and $ (\psi,\vv) $ around $ y= \varphi(x) $ such that
	\begin{align} \label{e3}
	&	\phi (\uu) =  E_1 +  E_2  \subseteq {\bf E_1} \oplus {\bf E_2}, \nonumber\\ 
	&	\psi (\vv) =  F_1 +  F_2  \subseteq {\bf F_1} \oplus {\bf F_2}, \nonumber \\
	&	\varphi_{\phi\psi} = \overline{\varphi} + \hat{\varphi} \circ \pr_{E_2},
	\end{align}
		where $ \overline{\varphi} : E_1 + E_2 \to F_1 $ is an \mc-mapping, $ \hat{\varphi} $
	is an \mc-isomorphism of $ \bf E_2 $ onto $ \bf F_2 $ and $ \pr_{E_2}: E \to {\bf E_2} $
	is the projection.
	
	Let $ \hat{e} \in \varphi^{-1}(S) \cap \uu $, then $ \hat{f} = \varphi(\hat{e}) \in S \cap \vv $
	and $ \psi (\varphi(\hat{e})) \in F_1 \subset \mathbf{F_1}$. By  \eqref{e3},
	if $ \phi(\hat{e}) = e_1+e_2 \in E_1+E_2 $ we have 
	\begin{align*}
		\varphi_{\phi \psi}(\phi(\hat{e})) &= 	\varphi_{\phi \psi}(e_1+e_2) \\
		&= \overline{\varphi} (e_1+e_2) + \hat{\varphi}(e_2) \in F_1 \subset \mathbf{F_1}.
	\end{align*}
	It follows $ \hat{\varphi}(e_2)=0_E, e_2=0_E, $ since $ \hat{\varphi}(e_2) \in \mathbf{F_2} $ and
	$ \mathbf{F_1} \cap \mathbf{F_2} = \set{0_F} $.
	
	Thus, $\phi (\hat{e}) \in F_1$ for all $ \hat{e} \in \varphi^{-1}(S) \cap \uu$. Therefore,
	$ E_1 \subset \phi ( \varphi^{-1}(S) \cap \uu) $, since for each $ e_1 \in E_1 $ we have
	$$
		\varphi_{\phi \psi}(e_1) = \overline{\varphi}(e_1)+ \hat{\varphi} \circ \pr_{E_2}(e_1) = \overline{\varphi}(e_1) \in F_1.
	$$
	Hence, $ \psi \circ \varphi \circ \phi^{-1}(e_1) \in F_1  $ implies that 
	$ \varphi \circ \phi^{-1}(e_1) \in \psi^{-1}(F_1) =S \cap \vv $ and so
	$ \varphi \circ \phi^{-1}(e_1) $ which means $ \phi^{-1}(e_1) \in \varphi (S) \cap \vv$
	that yields $ e_1 \in \psi (\varphi^{-1}(S) \cap \vv) $.
	Therefore, for $x \in \varphi^{-1}(S) $ there is a chart $ (\phi,\uu) $ with $ \phi (\uu) = E_1+E_2
	\subset \mathbf{E_1} \oplus \mathbf{E_2} $ and $ \phi(x) =0_E , \, \phi (\varphi^{-1}(S) \cap \vv) =E_1,$
	which means $ \varphi^{-1}(S) $ is an \mc-submanifold in $ M $.
	
	In the charts, we have $ T_x \simeq E, \, T_y N \simeq F,\, T_x (\varphi^{-1}(S)) \simeq \mathbf{E_1} $ and
	$ T_yS \simeq \mathbf{F_1} $. From the proof of Proposition \ref{p1} we have
	$$
   \varphi_{\phi \psi}'(0_E)^{-1}(\mathbf{F_1}) =\mathbf{E_1}
	$$
	which yields $ (T_x \varphi)^{-1}(T_yS) = T_x (\varphi^{-1}(S)) $.
	
	If $ S $ has finite co-dimension then $ \mathbf{F_2} $ has finite dimension  and thus by Proposition \ref{p1},
	$$
	\codim (\varphi^{-1}(S)) = \codim \varphi^{-1}(S \cap \vv) = \dim (\mathbf{F_2})=\codim (S).
	$$
	The proof of the last statement is standard. 
\end{proof}
As an immediate consequence we have:
\begin{cor}\label{s1}
	Let $\varphi: M \rightarrow N$ be an \mc-mapping, $k\geq1$. If $q$ is a regular value of  $\varphi$,
	then the level set $\varphi^{-1}(q)$ is a submanifold  of $M$ and its tangent space at $p=\varphi(q)$ is $\ker T_p\varphi$.  Moreover, 
	if $q$ is a regular value of $ \varphi $  and $\varphi$ is an \mc-\li-Fredholm mapping of index $ l $, 
	then $\dim \varphi^{-1}(S) = l$. 
\end{cor}
To prove the parametric transversality theorem we apply the following Sard's theorem.
\begin{theorem}\cite[Theorem 3.2]{k2}\label{sard}
	If $ \varphi: M \rightarrow N $ is an
	\mc-Lipschitz-Fredholm map with $  k > \max \lbrace {\Ind \varphi,0} \rbrace $. Then, the set of regular
	values of $ \varphi$ is residual in $ N $.
\end{theorem}
\begin{theorem}[The Parametric Transversality Theorem]\label{p}
	Let $ A $ be a manifold of dimension $ n $, $ S \subset N $ a submanifold of finite co-dimension $ m $.
	Let $ \varphi: M \times A \to N $ be an \mc-mapping, $ k\geq \set{1, n-m} $. If $ \varphi \pitchfork S $,
	then the  set of all points $ x \in M $ such that the mappings $$ \varphi_{x}: A \to N,\, (\varphi_{x}(\cdot)\coloneq \varphi (x,\cdot)) $$ are transversal to $ S $, is residual $ M $.
\end{theorem}
\begin{proof}
Let $ \mathbf{S} = \varphi^{-1}(S) $, $ \pr_{M}: M \times A \to M $ the projection onto $ M $
and $ \pr_{\mathbf{S}} $ be its restriction to $ \mathbf{S} $. First, we prove that 
$ \pr_{\mathbf{S}} $ is an \mc-Fredholm-Lipschitz mapping of index $ n-m $, i.e., 
$$
T_{(m,a)}\pr_{\mathbf{S}} : T_{(m,a)}\mathbf{S} \to T_mM
$$
is a \li-Fredholm operator of index $ n-m $.

By Theorem \ref{t1} the inverse image $ \mathbf{S} $ is an \mc-submanifold of $ M \times A $, with model space $ \mathbb{S} $,
so that  $ \pr_{\mathbf{S}} $ is an \mc-mapping.

Let $ \pi_M $ and $ \pi_{\mathbf{S}} $ be the local representatives of $ \pr_{M} $ and $ \pr_{\mathbf{S}} $, respectively. We show that $ \pi_{M}$ and consequently  $ \pi_{\mathbf{S}} $ are \li-Fredholm operators of index $ n-m $.

Finite dimensionality of $ \rr^n $ and closedness of $ \mathbb{S} $ implies that 
$ K \coloneq \mathbb{S} + (\{0\}\times \rr^n) $ is closed in $ E \times \rr^n $.
Also, $ \codim K $ is finite because it contains the finite co-dimensional subspace $ \mathbb{S} $.
Therefore $ K $ has a finite-dimensional complement $ K_1 \subset E \times \set{0} $, that is $ E \times \rr^n = K \oplus K_1 $. 
Let $ K_2 \coloneq \mathbb{S} \cap \set{0} \times \rr^n $. Since $ K_2 \subset \rr^n $
we can choose closed subspaces $ \mathbb{S}_1 \subset \rr^n $ and $ \rr_0 \subset \set{0} \times \rr^n $
such that $ \mathbb{S} = \mathbb{S}_1  \oplus K_1$ and $ \set{0} \times \rr^n = K_1 \oplus \rr_0 $.
Whence, $ K = \mathbb{S}_1 \oplus K_1 \oplus \rr_0 $ and
$ E \times \rr^n  = \mathbb{S}_1 \oplus K_1 \oplus \rr_0 \oplus K_2$.

The mapping $ \pi_{\mathbf{S}} \mid _{\mathbb{S}_1 \oplus K_2}  : \mathbb{S}_1 \oplus K_2 \to E$
is an isomorphism, $ K_1 = \ker \pi_{\mathbf{S}} $, and 
$ \pi_M (K_2) $ is a finite dimensional complement to $ \pi_M(\mathbb{S}) $ in $ \rr^n $. Thus,
$ \pi_M $ is a \li-Fredholm operator and we have
\begin{align*}
	\Ind \pi_M &= \dim K_1 - \dim K_2 \\
	&= \dim (K_1 \oplus \rr_0) - \dim (\rr_0 \oplus K_2).
\end{align*}
Since, $ K_1 \oplus \rr_0 = \set{0} \times \rr^n $ and $ \rr_0 \oplus K_2 $ is a complement to
$ \mathbb{S} $ in $ E \times \rr^n $ and therefore its dimension is $ n $, so the index of $ \pi_M $
is $ n-m $.

Now, we prove that if $ x $ is a regular value of $ \pr_{\mathbf{S}} $ if and only if $ \varphi_{x} \pitchfork S $.
From the definition of $ \varphi \pitchfork $ we have $\forall (x,a) \in \mathbf{S}$
\begin{equation}\label{1}
(T_{(x,a)} \varphi)(T_x M \times T_a A ) + T_{\varphi(x,a)}S = T_{\varphi(x,a)}N,
\end{equation}
and
\begin{equation}\label{2}
(T_{(x,a)}\varphi)^{-1}(T_{\varphi_{(x,a)}}S)\, \mathrm{splits \, in}\,  T_{x}M \times T_{a}A. 
\end{equation}
Since $ A$ has finite dimension, it follows that the mapping $ a \in A \mapsto \varphi{(x,a)} $ for a fixed $ x \in M $
is transversal to $ S $ if and only if 
\begin{equation}\label{5}
\forall (x,a) \in \mathbf{S}, T_a \varphi_x (T_aA)+ T_{\varphi(x,a)}S = T_{\varphi(x,a)}S.
\end{equation}
Since $ \pr_{\mathbf{S}} $ is a \li-Fredholm mapping, $ \ker T \pr_{\mathbf{S}} $ splits at any point as its dimension 
is finite. Then $ x $ is a regular value of $ \pr_{\mathbf{S}} $ if and only if 
\begin{equation}\label{3}
\forall (x,a) \in \mathbf{S}, \forall v \in T_xM, \exists u \in T_aA \colon T_{(v,u)} \varphi (v,u)\in T_{(x,a)}S. 
\end{equation}
Pick $ x \in M $ and $ a \in A $ such that $ (x,a) \in \mathbf{S} $ and let $ w \in T_{(x,a)}S $. By
\eqref{1} and \eqref{2} we obtain that there exist $ v \in T_aA,\, x_1 \in T_xM ,\, y_1 \in T_{(x,a)}S$
such that
\begin{equation}\label{4}
T_{(x,a)}\varphi(v,x_1) + y_1 = w.
\end{equation}
Then, there exists $ x_2 \in T_xM $ such that $ T_{(x,a)}\varphi (v,x_2) \in T_{\varphi_{(x,a)}}S $.
Hence, 
\begin{align*}
	w &= T_{(x,a)}\varphi(v,x_1) - T_{(x,a)}\varphi(v,x_2) + T_{(x,a)}\varphi(v,x_2) +y_1 \\
	  &= T_{(x,a)}\varphi(0,x_1-x_2) + T_{(x,a)}\varphi(v,x_2) +y_1 \\
	  &= T_{(x,a)}\varphi(0,u) + T_{\varphi(x,a)}S +y_2 \in T_a \varphi_x(T_aA),
\end{align*}
where $ u =x_1 -x_2 $ and $ y_2 = T_{(x,a)}\varphi (v,x_2) +y_1 \in T_{\varphi_{(x,a)}}S $. Thus,~\eqref{5} holds.

Now we show that \eqref{5} implies \eqref{3}. Pick $ a \in A,\, x \in M $ such that $ (x,a) \in \mathbf{S} $.
Let $ v \in T_xM,\, a_1 \in T_aA, \, y_1 \in T_{\varphi_{(x,a)}}S $ and set $ w \coloneq T_{(x,a)} \varphi_{(v,x_1)}+y_1 $. By \eqref{5} there exist $ a_2 \in T_aA $ and $ y_2 \in T_{\varphi_{(x,a)}}S $
such that $ w = T_a\varphi_x(a_2)+ y_2 $.
Then,
$$
0_E = T_{(x,a)}\varphi (v, a_1)-T_a\varphi_x(a_2)+y_1-y_2 = T_{(x,a)}\varphi (v, a_1-a_2) + y_1-y_2,
$$
so $ T_{(x,a)}\varphi (v, a_1-a_2) = y_2-y_1 \in T_{\varphi_{(x,a)}}S $ so \eqref{3} holds.
Thus, we showed  that if $ x $ is a regular value of $ \pr_{\mathbf{S}} $ if and only if $ \varphi_{x} \pitchfork S $.
Since $ \pr_{\mathbf{S}} : \mathbf{S} \to M $ is a \li-Fredholm of class \mc with the index $ n-m $ and
$ \codim \mathbf{S} = \codim S =m $ and $ k > \set{0, n-m} $, the Sard's theorem \ref{sard} concludes the theorem.
\end{proof}
\section{The degree of \li-Fredholm mappings}
In this section we construct the degree of \mc-\li-Fredholm mappings and apply it to prove an invariance of domain theorem, a rank theorem and a Bursuk-Ulam type theorem. The construction of the degree relies on the following transversality result.
\begin{theorem}\cite[Theorem 3.3]{k2} \label{transv}
	Let $\varphi: M \to N$ be an \mc-Lipschitz-Fredholm mapping, $k \geq 1$. 
	Let $\imath: \ab \to N$ be an $MC^1$-embedding of a finite dimension manifold $\ab$
	with $k > \max \{ \Ind \varphi + \dim \mathcal{A} ,0 \}$. Then
	there exists an $MC^1$ fine approximation $\mathbf{g}$ of $\imath$ such that $\mathbf{g}$ is embedding and  $\varphi \pitchfork \mathbf{g}$.
	Moreover, suppose $S$ is a closed subset of $\ab$ and $\varphi \pitchfork \imath(S)$, then 
	$\mathbf{g}$ can be chosen so that $\imath = \mathbf{g}$ on $S$.
\end{theorem}
We shall need the following theorem that gives the connection between proper and closed mappings.
\begin{theorem}\cite[Theorem 1.1]{as}\label{impo}
	Let $A, B$ be Hausdorff manifolds, where $A$ is a connected infinite
	dimensional Fr\'{e}chet manifold, and $B$ satisfies the first countability axiom, and let
	$\varphi: A \to B$ be a
	continuous closed non-constant map. Then $\varphi$ is proper.
\end{theorem}

Let $ \varphi: M \to N $ be a non-constant closed \li-Fredholm mapping with index $ l \geq 0$ of class \mc such that 
$k > l+1  $. If $ q $ is a regular value of $ \varphi $, then by Theorem~\ref{impo} and Corollary~\ref{s1} the preimage $ \varphi^{-1}(q) $
is a compact submanifold of dimension $ l .$ 

Let $ \imath: [0,1] \hookrightarrow N $ be an $ MC^1 $-embedding that connects two distinct regular values
$ q_1 $ and $ q_2 $. By Theorem~\ref{transv} we may suppose $ \imath $ is transversal to $ \varphi $.
Thus, by Theorem~\ref{t1} the preimage $ {\bf M} \coloneq \varphi^{-1} ( \imath ([0,1])) $
is a compact $ (l+1) $-dimensional submanifold of $ M $ such that its boundary, $ \partial {\bf M} $,
is the disjoint union of $ \varphi^{-1}(q_1) $ and $ \varphi^{-1}(q_2) $, 
$ \partial {\bf M} = \varphi^{-1}(q_1) \amalg \varphi^{-1}(q_2) $.
Therefore, $ \varphi^{-1}(q_1) $ and $ \varphi^{-1}(q_2) $ are non-oriented cobordant which gives the invariance of
the mapping. Following Smale~\cite{smale} we associate to $ \varphi $ a degree, denoted by $ \deg \varphi $,
defined as the non-oriented cobordism class of  $ \varphi^{-1}(q) $ for some regular value $ q $.
If $ l=0 $, then $ \deg \varphi \in \Z_2 $ is the number modulo 2 of preimage of a regular value.

Let $ \mathcal{O} \opn M $. Suppose $ \varphi : \overline{\mathcal{O}} \to N $
is a non-constant closed continuous mapping such that its restriction to $ \mathcal{O} $
is an $ MC^{k+1} $-\li-Fredholm mapping of index $ k $, $ k\geq0 $.
Let $ p \in N \setminus \varphi (\partial \overline{\mathcal{O}}) $ and let $ \mathbf{p} $ a regular value of $ \varphi $
in the connected component of $  N \setminus \varphi (\partial \overline{\mathcal{O}}) $ containing $ p $,
the existence of such regular value follows from Sard's theorem \ref{sard}. Again, we associate to $ \varphi $ a degree,
$\deg (\varphi,p)$, defined as non-oriented class of $ k $-dimensional compact manifold 
$ \varphi^{-1}(\mathbf{p}) $. This degree does not depend on the choice of $ \mathbf{p} $.

The following theorem which presents the local representation of \mc-mappings is crucial for the rest of the paper.
\begin{theorem}\cite[Theorem 4.2]{k1} \label{rp}
	Let $ \varphi : \uu \opn E \rightarrow F $ be an $ MC^k $-mapping, $ k \geq 1 $, $ u_0 \in \uu $. Suppose that 
	$\dd \varphi(u_0)$ has closed split image $ \bl {F_1} $ with closed topological complement $\bl{F_2}$ 
	and split kernel $ \bl{E_2}$ with closed topological complement $ \bl {E_1}$. Then, there are two open sets
	$ \uu_1 \opn \uu $ and
	$ \vv \opn \bl{F_1} \oplus \bl{E_2} $ and an $ MC^k$-diffeomorphism $ \Psi : \vv \rightarrow \uu_1 $,
	such that $ (\varphi \circ \Psi)(f,e) = (f,\eta (f,e))$ for all $ (f,e) \in \vv $, where $ \eta : \vv \to \bl{E_2} $ 
	is an $ MC^k$- mapping.
\end{theorem}
\begin{theorem}[Rank theorem for \mc-mappings]
	Let $\varphi : \uu \opn E \to F$ be an \mc-mapping, $k\geq 1$.
	Suppose $u_0 \in \uu$ and $\dd \varphi(u_0)$ has  closed split image $\bl{F_1}$
	with closed complement $\bl {F_2}$ and split kernel $\bl{E_2}$ with closed complement $\bl{E_1}$.
	Also, assume $\dd \varphi (\uu)(E)$ is closed in F and 
	$\dd \varphi(u)|_{\bl{E_1}} : \bl{E_1} \to \dd \varphi(u)(E)$ is an \mc-isomorphism for each $u \in \uu$.
	Then, there exist open sets $\uu_1 \opn \bl{F_1} \oplus \bl{E_2}, \, \uu_2 \opn E, \, \vv_1 \opn F$,
	and $\vv_2 \opn F$ and there are \mc-diffeomorphisms $\phi: \vv_1 \to \vv_2$ and $\psi: \uu_1 \to \uu_2$ such that
	$$
	(\phi \circ \varphi \circ \psi) (f,e) = (f,0), \quad \forall (f,e) \in \uu_1.
	$$
\end{theorem}

\begin{proof}
	By Theorem \ref{rp} there exits an \mc-diffeomorphism
	$
	\psi : \uu_1 \opn \bl{F_1} \oplus \bl{E_2} \to \uu_2 \opn E
	$
	such that
	$$
	\varphi(f,e) = (\varphi \circ \psi)(f,e) = (f, \eta (f,e)),
	$$
	where $ \eta : \vv \to \bl{E_2} $ 
	is an $ MC^k$- mapping. Let $\mathrm{Pr}_1 : F \to \bl{F_1}$ be the projection.
	We obtain $\mathrm{Pr}_1 \circ \dd \varphi (f,e)(w,v) = (w,0)$, for $w \in \bl{F_1}$ 
	and $v \in \bl{E_2}$ because
	$$
	\dd \varphi (f,e)(w,v) = (w, \dd \eta (f,e)(w,v)).
	$$
	Hence, $\mathrm{Pr}_1 \circ \dd \varphi (f,e)|_{\bl{F_1}\times \set{0}}$ is the identity mapping, $\id_{\bl{F_1}}$,
	on $\bl{F_1}$. Thereby,
	$$
	\dd \varphi (f,e)|_{\bl{F_1}\times \set{0}} : \bl{F_1}\times \set{0} \to \dd \varphi (f,e) (\bl{F_1} \oplus \bl{E_2})
	$$
	is one-to-one and therefore by our assumption 
	$$
	\dd \varphi (f,e) \circ \pr_1 |_{\dd \varphi (f,e)(\bl{F_1} \oplus \bl{E_2})}
	$$
	is the identity mapping. Suppose $(w, \dd \eta (f,e)(w,v)) \in \dd \varphi (f,e)(\bl{F_1} \oplus \bl{E_2}) $,
	we obtain $\dd \eta (f,e)v =0$ for all $v \in \bl{E_2}$, which means $\dd_2\eta(f,e)=0$, since
	\begin{align*}
		(\dd \varphi (f,e) \circ \pr_1)(w, \dd \eta (f,e)(w,v)) =&  \dd \varphi (f,e)(w,0) \\
		                                                        =& (w, \dd \eta (f,e)(w,0))\\
		                                                        =&(w,\dd_1 (f,e)w).  
	\end{align*} 
	We have $\dd^2 \varphi(f,e)v = (0, \dd_2 \eta (f,e)v)$, i.e., $\dd^2 \varphi (f,e) = 0$
	which means $\varphi$ does not depend on the variable $y \in \bl{E_2}$. Let $ \pr_2 : \bl{F_1} \oplus \bl{E_2} \to\bl{F_1} $ be the projection and
	$
	\varphi_f \coloneq \varphi (f,e) = (\varphi \circ \pr_2)(f,e),
	$
	so that $\varphi_f : \pr_2 (\uu_1) \subset \bl{F_1} \to F$.
	
	Let $ \bl{v_0} \coloneq (\pr_2 \circ \psi^{-1})(u_0) $ and $\vv \opn \uu$ be an open neighborhood of $\bl{v_0}$.
	 Define the mapping
	\begin{equation*}
	\begin{array}{cccc}
	\Phi : \vv \times \bl{F_2} \to \bl{F_1} \oplus \bl{F_2}\\
	\Phi(f, e) = \varphi (f) + (0,e).
	\end{array}
	\end{equation*}
	By the open mapping theorem 
	$$
	\dd \Phi (\bl{v_0},0) = (\dd \varphi (\bl{v_0}), \id_{\bl{F_2}}) : E \oplus \bl{F_2} \to F
	$$
	is a linear \mc-isomorphism, where $  \id_{\bl{F_2}}$ is the identity mapping of $\bl{F_2}$.
	Now $\Phi$ satisfies the inverse function theorem \ref{invr} at $\bl{v_0}$, therefore,
	there exist $\vv_1, \vv_2 \opn F$ such that $(\bl{v_0},0) \in \vv_2$ and 
	$\Phi (\bl{v_0},0) = \varphi_{\bl{v_0}}(\bl{v_0}) \in \vv_1$ and an \mc-diffeomorphism
	$\phi: \vv_1 \to \vv_2$ such that $\phi^{-1} = \Phi |_{\vv_1}$. Thus, for $ (f,0) \in \vv_2 $ we have
	$$
	(\phi \circ \varphi)(f) = (\varphi \circ \Phi)(f,0)= (f,0),
	$$
	and therefore, 
		$$
	(\phi \circ \varphi \circ \psi) (f,e) = (f,0), \quad \forall (f,e) \in \uu_1.
	$$
\end{proof}
As an immediate consequence we have the following:
\begin{cor}[Rank theorem for \li-Fredholm mappings]\label{li-ra}
	 Let $\varphi : M \to N$ be
	an $MC^{\infty}$-\li-Fredholm mapping of index $k$ and 
	$\dim \ker \dd \varphi (x) = m$, $\forall x \in M$. Let $\complement_1, \complement_2$ be topological complements of $\rr^m$ in $E$ and $\rr^{m-k}$ in $F$, respectively. Then, there exist charts
	$\phi : \uu \opn M \to E= \rr^m \oplus \complement_1$ with $\phi(x)=0_E$ 
	and $\psi :\vv \opn N \to F = \rr^{m-k} \oplus \complement_2$ with $\psi(x) = 0_F$ such that 
	$$
	\psi \circ \varphi \circ \phi^{-1}(f,0) = (f,0).
	$$
\end{cor}

The following theorem gives the openness property of the set of \li-Fredholm mappings.
\begin{theorem}\cite[Theorem 3.2]{k1}\label{open}
The set $ \mathcal{LF}(E,F) $ is open in $\mathcal{L}_{d,g}(E,F)$ with respect to the topology defined by the metric
\eqref{metric}. Furthermore, the function $ T \rightarrow \Ind T $ is continuous on $ \mathcal{LF}(E,F) $,
hence constant on connected components of $ \mathcal{LF}(E,F) $.
\end{theorem}
The proof of the following theorem is a minor modification
of~\cite[Theorem 2]{tr}.

\begin{theorem}\label{injec}
	Let $ \varphi: M \to N $ be a Lipschitz-Fredholm mapping of class  \mc, $k\geq 1$.
 Then, the set $\mathsf{Sing}(\varphi) \coloneq \set{m \mid \dd \varphi (m) \mathrm{is \, not \, injective}}$
 is nowhere dense in $ M $.
\end{theorem}
\begin{proof}
	This is a local problem so we assume $ M $ is an open set in $ E $ and 
	$ N $ is an open set in $ F $. Let $ s \in \mathsf{Sing} (\varphi) $ be arbitrary and $ \uu $
	an open neighborhood of $ s $ in $ \mathsf{Sing} (\varphi)$. For each 
	$ n \in \nn \cup \set{0} $ define
	\begin{equation*}
	S_n \coloneq \set{m \in M \mid \dim \dd \varphi (m) \geq n}.
	\end{equation*}
	Then, $ M = M_0 \supset M_1 \supset \cdots ,$ therefore, is a unique
	$ n_0 $ such that $ M = M_{n_0} \neq M_{n_0+1} $. Let $ m_0 \in M_{n_0} \setminus M_{n_{0}+1}$ 
	such that $ \dim \ker \dd \varphi (m_0) =n_0 $.
	By Theorem \ref{open}, there exists an open \nbd $ \vv$ of $ m_0 $ in $ \uu $
	such that for all $ v \in \vv $ we have $ \dim \ker \dd \varphi (v) \leq n_0 $ and hence 
	$ \dim \dd \varphi(v) = n_0 \geq 1 $.
	By Corollary \ref{li-ra}, there is a local representative $ \varphi $
	around zero such that
	$ \psi \circ \varphi \circ \phi^{-1}(f,e) = (f,0) $ for $ (f,e) \in \complement_1 \oplus \rr^{n_0} $
	which contradicts the injectivity of $ \varphi $, therefore, $ \mathsf{Sing}(\varphi) $ contains a nonempty open set.  The closedness of $\mathsf{Sing} (\varphi)$ is obvious in virtue of Theorem \ref{open}.
\end{proof}	
\begin{theorem}[Invariance of domain for \li-Fredholm mappings]\label{in-do}
	Let $ \varphi: M \to N  $ be an \mc-\li-Fredholm mapping of index zero, $ k>1 $.
	If  $ \varphi $ is locally injective, then $ \varphi $ is open.
\end{theorem}
\begin{proof}
	Let $ p \in U \opn M $ and $ q= \varphi (p)$. 
	The point $ p $ has a connected open \nbd $\, \uu \opn M $ such that 
	$ \varphi \mid_{\overline{\uu}} : \overline{\uu} \to N $ is proper and injective. 
	Whence $ q \notin \varphi (\partial \uu) $ and $ \varphi (\partial \uu) $ is closed in $ N $.
	Let $ \vv $ be a connected component of $ N \setminus \varphi (\partial \uu)$
	containing $ q $ which is its open neighborhood. Since $ \uu $ is connected it implies that $ \varphi (\uu) \subset \vv $.
	It follows from $ \varphi (\partial \uu)  \cap \vv = \emptyset$ that $ \overline{\uu} \cap \varphi^{-1}(\vv) = \uu $
	and so $ \varphi \mid_{\uu} : \uu \to N $ is proper and injective.  
	By Theorem~\ref{injec} there is a point $ x \in M $ such that the tangent map 
	$ T_x\varphi $ is injective and since $ \Ind \varphi = 0 $ it is surjective too. 
	Therefor, $ y = \varphi(x) $ is a regular value with $ \varphi^{-1}(y) = \set{x} $
	and so $ \deg \varphi = 1 $. It follows that $ \varphi $ is surjective, because
	if it is not , then any point in $ N \setminus \varphi (M)$  is regular and $ \deg\varphi = 0 $
	which is contradiction. Then, $ \vv = \varphi (\uu) $ is the open \nbd of $ q $.   
\end{proof} 
\begin{cor}[Nonlinear Fredholm alternative]
	Let $ \varphi: M \to N  $ be an \mc-\li-Fredholm mapping of index zero, $ k>1 $.
	If $ N $ is connected and $ \varphi $ is locally injective, then $ \varphi $
	is surjective and finite covering mapping. If $ M $
	is connected and $ N $ is simply connected, then $ \varphi $ is a homeomorphism.
\end{cor}
The following theorem is a generalization of the Bursuk-Ulam theorem, the proof is slight modification 
of the Banach case.  
\begin{theorem}
	Let $ \varphi : \overline{\uu} \to F $ be a non-constant closed \li-Fredholom mapping of class 
	$ MC^2 $ with index zero, where $ U \opn F  $ is  symmetric.
	If $ \varphi $ is odd and for $ u_0 \in \overline{U} $ we have $ u_0 \notin \varphi (\partial \overline{\uu}) $.
	Then $ \deg (\varphi, u_0) \equiv 1 \mod 2. $ 
\end{theorem}
\begin{proof}
	Since $\dd \varphi (u_0) $ is a \li-Fredholm mapping with index zero $$ F = \mathbf{F_1} \oplus \ker \varphi = \mathbf{F_2} \oplus \Img \varphi$$ and $ \dim \mathbf{F_2} = \dim \ker \varphi$. The image $ \varphi(\partial \uu) $ is closed as $ \varphi $
is closed, hence $$ \mathbf{a} = \varrho ( \varphi (\uu), u_0) > 0 $$ because $ u_0 \notin \varphi (\partial \uu) $.

Let $ \phi : F \to F$ be a global \li-compact linear operator with  $\mathpzc{Lip}(\phi) < \mathbf{b}$ for some
$ \mathbf{b}>0 $. Define the mapping 
$ \Phi_{\phi} : \overline{U} \to F$ by 
$\Phi_{\phi}(u) = \varphi (u) + \phi (u)$.
Then $ \Phi_{\phi}$ is a \li-Fredholm mapping of index zero. Suppose $ \mathbf{b} < \mathbf{a} /\mathbf{k} $ for some $ \mathbf{k} > 1$, then
\begin{equation*}\label{key}
\varrho (\Phi_{\phi} (u), u_0) \geq \varrho (\varphi (e), u_0) - \mathpzc{Lip}(\phi) \varrho (U,u_0) > \mathbf{a} - \mathbf{bk} > 0,  \quad \forall u \in \partial \uu.
\end{equation*}
Therefore, $ u_0  \notin \Phi_{\phi} (\partial \uu)$. We obtain $ \deg (\varphi, u_0) = \deg (\Phi_{\phi}, u_0)$ 
as the mapping $$ \psi: [0,1] \times \overline{U} \to F $$ defined by $ (t,u) \to \varphi (u) +t \phi(u) $
is proper and $ u_0 \notin \psi (\partial U) $ for all $ t $. Considering the fact that
$ \psi (-u) = -\psi (u) $, we may use the perturbation by compact operators to find the degree of $ \varphi $.
Let $ \mathsf{C} $   be a set of global \li-compact linear operators  $ \phi: F \to F $  with  
$\mathpzc{Lip}(\phi) < \mathbf{b}  < \mathbf{a} /\mathbf{k}$. Let $ \widehat{\phi} \in \mathsf{C}$ be such that its restriction to 
$ \mathbf{F_1} $ equals $ u_0 $  and $ \widehat{\phi} \mid_{\ker \dd \varphi(u_0)} : \ker \dd \varphi(u_0) \to \mathbf{F_2} $ is an $ MC^1 $-isomorphism.
Therefore, $ \dd \varphi (u_0) + \widehat{\phi}$ and consequently $ \dd \varphi(u_0) $ is an $ MC^1 $-isomorphism. 
Now define the mapping $ \Psi : \uu \times \mathsf{C} \to F $ by $ (u,\phi) = \Phi_{\phi}(u) $. For sufficiently small
$ \mathbf{b} $ the differential $ \dd\Psi(u,\phi)(v,\psi) = (\dd \varphi (u) + \phi)v+ \psi(u)$ is surjective at $ u_0 $ as $ \dd \varphi (u_0) $ is an $ MC^1 $-isomorphism. Also, it is clear that it is surjective at the other points.
Then, the mapping $ \Psi $ satisfies the assumption of Theorem~\ref{p}, therefore, $ \Psi^{-1}(u_0) $
is a submanifold and the mapping $ \Pi : \Psi^{-1}(u_0) \to \mathsf{C} $ induced by
the projection onto the second order is \li-Fredholm of index zero. By employing the local version of Sard's theorem
we may find a regular point $ \overline{\phi} $ of $ \Pi $, and from the proof of the Theorem \ref{p} it follows that
$ u_0 $ is a regular value of $ \Phi_{\overline{\phi}} $ and consequently $ u_0 $ is a regular value of $ \varphi $.
Thus, properness and $ \varphi(-u)=-\varphi(u) $ imply that $  \varphi^{-1}(u_0) = \set{u_0, f_1,-f_1, \cdots f_m,-f_m }  $ and therefore 
$ \deg (\varphi, u_0) \equiv 1 \mod 2 $.

\end{proof}
\bibliographystyle{amsplain}

\end{document}